\documentclass[11pt]{amsart}
\usepackage{amssymb,color,a4wide}
\DeclareMathAlphabet{\mathpzc}{OT1}{pzc}{m}{it}

\newcommand{\w}{\omega}
\newcommand{\e}{\varepsilon}
\newcommand{\Sm}{\mathcal S\!\mathpzc m\hskip0.3pt}
\newcommand{\Sp}{\mathcal S\!\mathpzc p\hskip0.2pt}
\newtheorem{theorem}{Theorem}[section]
\newtheorem{proposition}[theorem]{Proposition}
\newtheorem{question}[theorem]{Question}
\theoremstyle{definition}
\newtheorem{example}[theorem]{Example}

\begin{document}

\title{Bornological, coarse and uniform groups}
\author{Igor  Protasov}
\begin{abstract} We survey and analyze different ways in which bornologies, coarse structures and  uniformities on a group agree with the group operations.
\end{abstract}
\address{Faculty of Computer Science and Cybernetics, Kyiv University,
         Academic Glushkov pr. 4d, 03680 Kyiv, Ukraine}
\email{i.v.protasov@gmail.com}

\maketitle

\section{ Bornological groups}
\normalsize

\subsection{Bornological spaces.}
A family $\mathcal{I}$ of subsets of a set $X$ is called an {\it ideal}  (in the Boolean algebra $\mathcal{P}_{X}$ of all subsets of  $X$)   if  $\mathcal{I}$ is closed under formation of finite unions  and subsets.
If $\bigcup \mathcal{I}=X$ then $\mathcal{I}$  is called  a {\it bornology}, so a bornology is an ideal containing the ideal $[X]^{<\w}$ of all finite subsets of $X$.

A {\em bornological space} is a pair $(X,\mathcal B)$ consisting of a set $X$ and a bornology $\mathcal B$ on $X$.
Any set $Y\in \mathcal{B} $  is called   {\it bounded}. If $X\in \mathcal{B}$,  then  the bornological space $(X, \mathcal{B}) $ is  {\it bounded}.

 Any subset $Y\subset X$ of a bornological space $(X,\mathcal B)$ carries the {\em subbornology} $$\mathcal B{\restriction}Y:=\{B\in\mathcal B:B\subset Y\},$$induced by $\mathcal B$.

A bornology  $\mathcal{B}$ on $ X $ is called
\begin{itemize}
\item {\em tall}  if $\mathcal B{\restriction}Y\ne[Y]^{<\w}$ for any infinite subset $Y$;
\item {\em antitall}  if any subset $Y\notin \mathcal{B}$ of $X$ contains an infinite subset $Z\subseteq  Y$ such that $\mathcal B{\restriction}Z=[Z]^{<\w}$.
\end{itemize}

By \cite[Proposition 1]{b13}, every bornology  is the  meet of some tall and antitall bornologies.

A family $\mathcal B'\subseteq\mathcal{B}$ is called a {\it base} of a bornology $\mathcal{B}$    if each set $B\in   \mathcal{B}$ is contained in some set $B'\in  \mathcal B'$. Every bornology  with a  countable  base is antitall. In particular, the
bornology  of all bounded subsets of a metric space is antitall.
We note also that for every bornology
  $\mathcal{B}$ on $X$  with a countable base, there exists  a metric $d$  on $X$ such that $\mathcal{B}$ is a bornology of all bounded subsets of $ (X, d)$.

The product of a family of bornological spaces $(X_\alpha,\mathcal B_\alpha)$, $\alpha\in A$, is the Cartesian product $\prod_{\alpha\in A}X_\alpha$ of their supports endowed  with the bornology generated by the base $\{\prod_{\alpha\in A}B_\alpha:(B_\alpha)_{\alpha\in A}\in\prod_{\alpha\in A}\mathcal B_\alpha\}$.

A mapping $f: X\to Y$ between two bornological spaces $(X,\mathcal{B}_X)$ and $(Y,\mathcal{B}_Y)$  is called {\it bornologous  } if $\{f(B):B\in\mathcal B_X\}\subset\mathcal B_Y$.

A {\it variety} is a class of
 bornological spaces closed under formation of subspaces,
 products and bornologous images.

We denote by $\mathfrak{M} _{single}$ the variety of all singletons,
$ \mathfrak{M}  _{bound}$ the variety of all bounded bornological spaces, $\mathfrak{M}_{\kappa}$ the variety of all  $\kappa$-bounded spaces.
For an infinite cardinal $\kappa$,  a bornological  space $(X, \mathcal{B})$  is called $\kappa$-{\it bounded}  if each subset $B\subset X$ of cardinality $|B|<\kappa$ is bounded.

Every variety of bornological spaces coincides with one of the varieties in the chain:
$$\mathfrak{M}  _{single} \subset \mathfrak{M}  _{bound}\subset \cdots\subset  \mathfrak{M}  _{\kappa}\subset\cdots\subset \mathfrak{M}  _{\omega},$$ see the proof of Theorem 2 in \cite{b14}.

\subsection{Bornologies on  groups}
A bornology on a group $G$  is called {\it right} ({\em left}) {\em invariant } if $\mathcal{B} g\subseteq \mathcal{B}$  (resp. $g\mathcal{B}\subseteq \mathcal{B}$) for each every $g\in G$. Here $\mathcal Bg=\{Bg:B\in\mathcal B\}$ and $g\mathcal B=\{gB:B\in\mathcal B\}$.
A group $G$ endowed with a right (left)  invariant bornology is called a {\it  right \textup{(}left\textup{)} bornological} group.

We say that a group $G$ endowed with a bornology $\mathcal{B}$ is a {\it bornological group} if the group multiplication and inversion are bornologous mapping.  In this case,  $\mathcal{B}$ is called   a {\it group bornology}.
 We note that $\mathcal{B}$  is a group bornology if and only if for any $A, B \in  \mathcal{B}$ we have
 $AB ^{-1} \in \mathcal{B}$. 

\subsection{Duality}
Now we endow every group $G$ with the discrete topology and identify the Stone-\v Cech compactification $\beta G$ of $G$ with the set of all ultrafilters on $G$.
  Then the family $ \{ \overline{A} : A \subseteq G \}$, where
  $\overline{A} = \{ p\in \beta G : A\in p \},$ forms a base of the topology of $\beta G$.
  Given a filter $\varphi$ on $G$, we denote $\overline{\varphi} = \cap \{ \overline{A} : A\in \varphi \}$, so $\varphi$  defines the closed subset   $\overline{\varphi}$ of   $\beta G$, and each  closed subset $K$ of  $\beta G$  can be obtained in this way: $K=\overline{\varphi} $,  where $ \overline{\varphi} = \{ A\subseteq G:  K \subseteq   \overline{A} \}$.

We use the standard extension \cite[Section 4.1]{b7} of the  multiplication on $G$ to the semigroup multiplication on  $\beta G$  such  that for every $p\in \beta G$ the right shift $\beta G\to\beta G$, $x\mapsto  xp$,  is continuous, and for every $g\in G$ the left shift $\beta G\to\beta G$, $x\mapsto gx$, is continuous. For ultrafilters $p, q\in \beta G$ their product $pq$ in $\beta G$ is defined by the formula $$pq=\big\{\mbox{$\bigcup_{x\in P}xQ_x$}:P\in p,\;\{Q_x\}_{x\in P}\subset q\big\}.$$

Let $G^{*}:=\beta G\setminus G$ be  the set of all free ultrafilters on $G$.
It follows directly from the definition of the multiplication in $\beta G$ that   $G^*$ and    $ \overline{ G^{*} G^{*}}$ are ideals  in the semigroup $\beta G$, and  $G^{*}$  is the unique  maximal closed ideal in  $G$. By Theorem 4.44 from \cite{b7},  the closure  $ \overline{ K(\beta G)}$  of the minimal ideal    $K(G)$  of   $\beta G$  is an ideal, so $ \overline{ K(\beta G)}$ is the smallest closed ideal in $\beta G$.  For the structure of    $ \overline{ K(\beta G)}$  and some other ideals in  $\beta G$  see \cite[Sections 4, 6]{b7}.

For an ideal $\mathcal{I}$ on a group $G$ and a closed subset $K$ of $\beta G$,
we put
$$\mathcal{I}^{\wedge} = \{p\in  \beta G:\forall A\in\mathcal I\;\; G\setminus A \in p\}\mbox{ and }
K^{\vee}=\{G\setminus A: A\in\varphi, \;\; \overline{\varphi}=K\}.$$
We have the following duality statements: \vskip 5pt
\begin{itemize}
\item $\mathcal{I}$ is left translation invariant if and only if $\mathcal{I}^{\wedge}$    is a left ideal of the semigroup $\beta G$ ; \vskip 5pt
\item  $\mathcal{I}$ is right  translation invariant if and only if
$(\mathcal{I}^{\wedge})G\subseteq  \mathcal{I}^{\wedge}$;
\item  $(\mathcal{I}^{\wedge})^{\vee}  =\mathcal{I}$;
\item  $\mathcal{I}$ is a bornology if and only if $\mathcal{I}^{\wedge}\subseteq G^{\ast}$.
\end{itemize}

Thus, we have the duality between left invariant bornologies on $G$ and closed left ideals of
$\beta G$ containing in $G^{\ast}$.
  We say that a subset $A$ of a group  $G$ is
 \begin{itemize}
\item{} {\it large} if  $G=FA$  for some $F\in [G]^{<\w}$;
\item{} {\it small} if  $L\setminus A$ is  large for every large subset $L$ of $G$;
\item{} {\it sparse} if  for every   infinite subset   $X$ of $G$ there exists a  finite subset $F\subset X$  such that $\bigcap _{g\in F}  gA$ is  finite.
\end{itemize}

\begin{theorem}\label{1.1} For every infinite group $G$, the family $\Sm _{G}$ of all small subsets of $G$ is a left and right invariant bornology and
$\Sm_{G}^{\wedge}= \overline{K(\beta G)}$.
\end{theorem}

This is Theorem 4.40 from \cite{b7} in the form given in \cite[Theorem 12.5]{b15}.

\begin{theorem}\label{1.2} For every infinite group $G$, the family $\Sp_{G}$ of all sparse subsets of $G$ is a left and right invariant bornology and   $\Sp^{\wedge}  _{G}=G^{\ast}G^{\ast}$.
\end{theorem}

This is Theorem 10 from \cite{b3}.

More applications of this duality can be find in \cite{b16}.

Let $\mathcal{B}$ be a group bornology on $G$. By \cite{b21},  $\mathcal{B}^{\wedge}$ is
an ideal in $\beta G$  but the converse statement does not hold: $\Sm^{\wedge} _{G}$  is an ideal but
$\Sm_{G}$  is not a group bornology.

\subsection{Plenty of bornologies} We say that a left invariant bornology $\mathcal{B}$ of $G$ is {\it maximal}  if $G\notin \mathcal{B}$ but $G\in\mathcal B'$ for every left
invariant bornology  $\mathcal{B}^{\prime}$ on $G$ such that  $\mathcal{B}\subsetneq \mathcal{B}^{\prime}$.

\begin{theorem}\label{1.3} For every infinite group $G$,  of cardinality $\kappa$, there are $2^{2^{\kappa}}$  distinct maximal left  invariant bornologies on $G$.
\end{theorem}

\begin{proof}
By Theorem  6.30 from \cite{b7},  there exists a family  $\mathfrak{F}$ consisting of $|\mathfrak{F}|=2^{2^{\kappa}}$ pairwise disjoint closed left ideals in  $\beta G$.
Each $I\in \mathfrak{F}$  contains some minimal closed left ideal $L  _{I}$.
Then  $\{L  _{I}^{\vee}: I\in\mathfrak{F}\}$
 is the desired family of  maximal left  invariant bo nologies on $G$.
\end{proof}

\begin{theorem}[\cite{b12}]\label{1.4}
Every countable group $G$ admits exactly $2^{\mathfrak c}$ group  bornologies.
\end{theorem}

By  Theorem 6.3.3 from \cite{b24},  each Abelian group $G$ admits exactly $2^{2^{|G|}}$  group bornologies. However this theorem does not extend to uncountable non-commutative groups. The following (consistent) counterexample was suggested by Taras Banakh.

\begin{example}\label{Banakh} Under CH there exists a group $G$ of cardinality $|G|=\mathfrak c=\w_1$ admitting exactly $2^\mathfrak c<2^{2^{\mathfrak c}}$ group bornologies.
\end{example}

\begin{proof} In \cite{Shelah} Shelah constructed a CH-example of a group $G$ of cardinality $|G|=\mathfrak c=\w_1$ such that $G=A^{6641}$ for any uncountable subset $A\subset G$. This implies that each group bornology $\mathcal B$ on  $G$ either coincides with $\mathcal P_G$ or is contained in the bornology $[G]^{\le\w}$ of countable subsets of $G$. So, the number of group bornologies of Shelah's group $G$ is $\le 2^{|[G]^{\le\w}|}=2^{\mathfrak c}$. On the other hand, Theorem~\ref{1.4} ensures that this number is $\ge 2^{\mathfrak c}$. Therefore, the number of group bornologies of Shelah's group $G$ is equal to $2^{\mathfrak c}$, which is strictly smaller than $2^{2^{\mathfrak c}}$ by known Cantor's Theorem.
\end{proof}

Theorem~\ref{1.4} implies that under the assumption $2^{\w_1}=2^\w$, each group $G$ of cardinality $|G|=\w_1$ has exactly $2^{\mathfrak c}=2^{2^{|G|}}$ group bornologies. This observation shows that the group  from Example~\ref{Banakh} cannot be constructed in ZFC.

\begin{theorem}[\cite{b16}]\label{1.5} For every infinite group $G$, the following statements hold:
\begin{itemize}
\item[(i)]  if  $\mathcal{B}$ is a left invariant bornology on $G$ and $\mathcal{B}\neq [G]^{<\w}$, then there is a
left invariant bornology  $\mathcal{B}^{\prime}$
 on $G$ such that  $[G]^{<\w}\subsetneq  \mathcal{B}^{\prime}\subsetneq \mathcal{B}$;
\item[(ii)]  if $G$  is either countable or Abelian  and $\mathcal{B}$ is a left  and right
 invariant bornology on $G$ such that
 $\mathcal{B}\neq [G]^{<\w}$
 then there is a  left and right  invariant bornology $\mathcal{B}^{\prime}$
 on $G$ such that $[G]^{<\w}\subsetneq   \mathcal{B}^{\prime}\subsetneq \mathcal{B}$;
 \item[(iii)]  if $G$  is the group $S_{\kappa}$  of all permutations of an infinite cardinal $\kappa$ then there exists a left and right   invariant bornology $\mathcal{B}$
 on $G$ such that
 $[G]^{<\w}\subsetneq   \mathcal{B}$
 and there are no left and right invariant bornologies on $G$ between $[G]^{<\w}$ and $\mathcal{B}$.
 \end{itemize}
 \end{theorem}

\begin{theorem}\label{1.6} Let  $G$ be an  infinite group and let $\mathcal{B}$ be  a group bornology of $G$ such that
 $\mathcal{B}\neq[G]^{<\w}$.
If $G$  is either countable or Abelian then there is a group bornology
$\mathcal{B}^{\prime}$ such that $[G]^{<\w}\subsetneq \mathcal{B}^{\prime}\subsetneq  \mathcal{B}$.
\end{theorem}

This is  Theorem 6.4.1.  from \cite{b24}.
We do not know (Question 6.4.1.   in  \cite{b24})  if Theorem \ref{1.6} remains true for every infinite group $G$,  in particular, for $G= S_{\kappa}$.
\smallskip

By  \cite[Theorem 12.9]{b15},  every infinite group can be partitioned into countably many small subsets.

\begin{question}\label{1.2} Given an infinite group $G$,  do there exist  a small subset $S$  and a countable subset $A$ of $G$  such that
$\{ aS: a\in A\}$ is a  covering (partition)  of $G$?
\end{question}

This is so if $G$ is amenable or $G$  has a subgroup of countable index.

\section{Coarse groups}

\subsection{Balleans and Coarse spaces}
For a set $X$ a subset $\e\subset X\times X$ containing the diagonal $\vartriangle _{X}:= \{(x,x) : x \in X\}$ is called {\em entourage} on $X$. For two entourages $\e,\e'$ on $X$ the sets
 $$
 \varepsilon\circ\e'=\{(x,z):\exists y\in X\;\,(x,y)\in\varepsilon,\;(y,z)\in\e'\}\mbox{  and  }
 \e^{-1}=\{(y,x):(x,y)\in\e\}.
$$
also are entourages on $X$.

 A family $\mathcal{E}$ of entourages on $X$ is called a {\it ball structure}  if it satisfies two axioms:
\begin{itemize}
\item[(a)]  for any $\varepsilon, \delta\in\mathcal{E}$ the entourage $\varepsilon\circ\delta^{-1}$ is contained in some $\lambda\in\mathcal E$;
\item[(b)]  $\bigcup\mathcal E=X\times X$.
\end{itemize}
A ball structure $\mathcal E$ on $X$ is called a {\em coarse structure} if it satisfies one additional axiom:
\begin{itemize}
\item[(c)] if $\varepsilon\in\mathcal{E}$ and $\vartriangle_{X}\subseteq \delta\subseteq\varepsilon$, then $\delta\in\mathcal{E}$.
\end{itemize}
It follows that each coarse structure is a ball structure and each ball structure $\mathcal E$ on $X$ is a base of the unique coarse structure
$${\downarrow}\mathcal E=\{\delta:\exists\e\in\mathcal E\;\vartriangle_X\subseteq\delta\subseteq \e\}.$$
A subfamily $\mathcal{E}^{\prime}\subseteq\mathcal{E}$ is called a {\it base} of a coarse structure $\mathcal{E}$ if for
every $\varepsilon\in\mathcal{E}$ is contained in some $\delta\in\mathcal{E}^{\prime}$.

A {\em ballean} is a pair $(X, \mathcal{E})$ consisting of a set $X$ and a ball structure $\mathcal E$ on $X$. If the ball structure $\mathcal E$ is a coarse structure, then the ballean $(X,\mathcal E)$ is called a {\em coarse space}.
 We note that coarse spaces can be considered as
 asymptotic counterparts of uniform spaces and balleans could
  be defined in terms of balls, see \cite{b15}, \cite{b24}.

Let $\mathcal E$ be a ball structure on a set $X$. For every
$x\in X$  a $\varepsilon\in\mathcal{E}$ the set
 $B(x, \varepsilon):= \{y\in X: (x,y)\in\varepsilon \}$ is called the {\it ball of radius  $\varepsilon$ centered at $x$.}

 A subset $Y$ of $X$ is called {\it bounded} if there exist $x\in X$ and $\varepsilon\in\mathcal{E}$
   such that $Y\subseteq B(x, \varepsilon)$.  The coarse
   structure
    $\mathcal{E}=\{\varepsilon\in X\times X: \bigtriangleup_{X}\subseteq\varepsilon\}$ is
     the unique coarse structure such that  $(X,\mathcal{E})$  bounded.

For a ballean $(X, \mathcal{E})$, each subset $Y \subseteq X$ carries the induced ball structure $\mathcal{E}{\restriction}{Y}:= \{\varepsilon\cap(Y\times Y): \varepsilon\in\mathcal{E}\}$. The ballen $(Y, \mathcal{E}{\restriction}{Y})$
 is called a {\it subballean} of $(X, \mathcal{E})$.

A subset $Y$ of a ballean $X$  is called {\it large} (or {\it  coarsely dense})  if there exists $\varepsilon\in \mathcal{E}$  such that $X= B(Y, \varepsilon)$ where $B(Y, \varepsilon)=\bigcup_{y\in Y} B(y, \varepsilon)$.

Let  $(X, \mathcal{E})$, $(X^{\prime}, \mathcal{E}^{\prime})$ be balleans.
A mapping
$f: X\to X^{\prime}$
is called {\it coarse} if for every
$\varepsilon\in\mathcal{E}$
there exists $\varepsilon^{\prime}\in\mathcal{E}'$ such that, for every $x\in X$,  we have
$f(B(x,\varepsilon))\subseteq B(f(x),\varepsilon^{\prime})$.
If $f$ is surjective and coarse, then $(X^{\prime}, \mathcal{E}^{\prime})$
 is called a {\it coarse image} of $(X, \mathcal{E})$.
If $f$ is a  bijection such that $f$  and $f^{-1}$ are coarse mappings,   then $f$ is called an
 {\it asymorphism}.
Two balleans
$(X, \mathcal{E})$, $(X^{\prime}, \mathcal{E}^{\prime})$
 are called {\it coarsely equivalent} if there exist large subsets
 $Y\subseteq X$, $Y^{\prime}\subseteq X^{\prime}$ such that the balleans
 $(Y, \mathcal{E}{\restriction}{Y})$
 and $(Y^{\prime}, \mathcal{E}^{\prime}{\restriction}{Y^{\prime}})$
  are asymorphic.

To conclude the coarse vocabulary, we take a family
$\{(X_{\alpha}, \mathcal{E}_{\alpha}) :  \alpha< \kappa\}$
  of balleans and define their
  {\it product}
 $\prod_{\alpha< \kappa}(X_{\alpha}, \mathcal{E}_{\alpha})$
  as the Cartesian product $\prod_{\alpha< \kappa} X_{\alpha}$
 endowed with the ball structure consisting of the entourages
 $$\{((x_\alpha),(y_\alpha)):\forall\alpha\in A \;(x_\alpha,y_\alpha)\in \e_\alpha\}$$where $(\e_{\alpha})_{\alpha\in A}\in\prod_{\alpha\in A}\mathcal E_\alpha$.

For lattices of coarse structures and varieties of coarse spaces, see \cite{b17} and \cite{b14}.

For every ballean $(X, \mathcal{E})$, the family of all bounded  subsets of  $X$  is a bornology.
On the other hand, for every bornology $\mathcal{B}$ on $X$, there is the   smallest by inclusion coarse structure $\mathcal{E} _{\mathcal{B}}$  on $X$ such that  $\mathcal{B}$ is the  bornology
 of all bounded subsets of $(X, \mathcal{E} _{\mathcal{B}})$.
A coarse structure $\mathcal{E}$  on $X$   is of the form  $\mathcal{E} _{\mathcal{B}}$  if and only if $(X, \mathcal{E})$  is {\it thin}: for every $\varepsilon \in\mathcal{ E}$,   there exists a bounded subset $A$  of  $(X, \mathcal{E})$
  such that $B(x,\varepsilon  )= \{ x\}$
for all $x\in X\setminus  A$. Thin coarse structures alternatively are called {\em discrete}.


\subsection{Coarse structures on groups} 
We remind that a bornology $\mathcal{I}$ on a group $G$ is  a {\it group bornology} if  and $AB^{-1}\in\mathcal{I}$
   for all $A,B\in \mathcal{I}$.
   A group bornology $\mathcal{I}$ is called {\it invariant} if
$\bigcup _{g\in G} g^{-1}A g\in\mathcal{I}$  for each $A\in \mathcal{I}$.

Let $X$ be a $G$-space with an action  $G\times  X\to X$, $(g,x)\mapsto gx$, of a group $G$.
We assume that $G$ acts on $X$ transitively,  take a group bornology
 $\mathcal{I}$ on $G$  and consider the coarse structure $\mathcal{E}(G, \mathcal{I}, X)$  on $X$ generated by the base consisting of the entourages $\varepsilon_{A}:=\{(x, y)\in X\times X:y\in\{x\}\cup Ax\}$. In this case $B(x,\varepsilon_A)=\{x\}\cup Ax$.

By \cite[Theorem 1]{b9}, for every coarse structure  $\mathcal{E}$ on $X$,
 there exist a group $G$ of permutations of $X$ and a group ideal $\mathcal{I}$ on $G$  such that
 $\mathcal{E}= \mathcal{E} (G, \mathcal{I}, X)$.

Now let $X=G$  and $G$  acts on $X$  by the left shifts.
We denote
$\mathcal{E}_{\mathcal{I}}= \mathcal{E} (G, \mathcal{I}, G)$.
Thus, every group bornology $\mathcal{I}$ on $G$ turns $G$ into the coarse space
$(G, \mathcal{E}_{\mathcal{I}})$.
We note that a subset $A$ of $G$ is bounded in $(G, \mathcal{E}_{\mathcal{I}})$ if and only if  $A\in\mathcal{I}$.

A group $G$ endowed with a coarse structure $\mathcal{E}$  is called {\it left (right) coarse group} if, for every $\varepsilon\in \mathcal{E}$,  there exists $\varepsilon^{\prime}\in \mathcal{E}$
 such that $g B(x,\varepsilon)\subseteq B(gx,\varepsilon^{\prime})$ (resp. $B(x,\varepsilon)g \subseteq B(xg,\varepsilon^{\prime})$~)
  for all $x, g\in G$.
  Equivalently,  $(G, \mathcal{E})$ is a left (right)  coarse group if  $\mathcal{E}$  has a  base consisting of left (right)   invariant entourages. An entourage $\varepsilon $  is  {\em left} ({\em right})  {\em invariant} if $g\varepsilon  =\varepsilon$ (resp. $\varepsilon g=g$)  for each $g\in G$,  where $g\varepsilon = \{(gx, gy): (x,y)\in \varepsilon)\}$ and $\varepsilon g = \{(xg, yg): (x,y)\in \varepsilon)\}$.
    For finitely generated groups, the right coarse groups $(G, \mathcal{E}_{[G]<\omega})$ in metric form paly an important role in
Geometrical Group Theory, see \cite[Chapter 4]{b5}.

A group $G$  endowed with a coarse structure $\mathcal{E}$  is called a
{\it coarse group} if the group multiplication $(G, \mathcal{E})\times (G, \mathcal{E})\to (G, \mathcal{E})$, $ (x,y)\mapsto xy$,  and the   inversion
$(G, \mathcal{E})\to (G, \mathcal{E})$, $x\mapsto x^{-1}$,
  are coarse mappings. In this case, $\mathcal{E}$ is called a {\it group coarse structure.}

The following two statements are taken from  \cite{b20}, see also \cite[Chapter 6]{b24}.

\begin{proposition}\label{2.1}
A group $G$ endowed with a coarse structure  $\mathcal{E}$ is a right coarse group if and  only if there exists a group bornology $\mathcal{I}$ on $G$ such that $\mathcal{E}=\mathcal{E}_{\mathcal{I}}$.
\end{proposition}

\begin{proposition}\label{2.2}
For a group $G$ endowed with a coarse structure  $\mathcal{E}$, the following conditions are equivalent:
\begin{enumerate}
\item[(i)]	$(G,\mathcal{E})$  is a coarse group;
\item[(ii)] $(G,\mathcal{E})$ is  left and right coarse  group;
\item[(iii)] there exists an invariant group bornology  $\mathcal{I}$ on $G$  such that $\mathcal{E}=\mathcal{E}_{\mathcal{I}}$.
\end{enumerate}
\end{proposition}

Applying Theorem 1.4, we get $2^{2^{\aleph_{0}}}$ distinct right coarse structures on any countable group.
For every infinite group $G$  and any infinite cardinal $\kappa\leq |G|$, the bornology  $[G]^{<\kappa}$  defines an unbounded right coarse structure  on $G$.
But if $G$ has only two conjugated  classes then there is  only one, bounded, group coarse structure on $G$.


\subsection{Asymorphisms.}
For an infinite cardinal $\kappa$, we say that two groups
$G$  and $H$ are $\kappa$-{\it asymorphic} (resp. $\kappa$-{\em coarsely equivalent}) if the right coarse structures on $G$  and $H$ defined by the bornologies $[G] ^{<\kappa}$  and  $[H] ^{<\kappa}$ is asymorphic  (resp. coarsely equivalent). In the case $\kappa=\aleph_{0}$,   $G$  and $H$ are  called {\em finitarily  asymorphic} and {\em  finitarily  coarsely equivalent}, respectively.

Let us recall that a group $G$ is {\em locally finite} if each finite subset of $G$ is contained in finite subgroup of $G$. A classification of countable locally finite groups  up to  finitary asymorphisms is obtained in \cite{b11} (cf.  \cite[p. 103]{b15}).

\begin{theorem} Two countable locally finite groups $G_{1}$ and $G_{2}$ are finitarily  asymorphic if and only if the following conditions hold:
\begin{itemize}
\item[(i)] for every finite subgroup $F_1\subset G_{1}$, there exists a finite subgroup $F_2$ of $G_{2}$ such that $|F_1|$ is a divisor of $|F_2|$;
\item[(ii)] for every finite subgroup $F_2$ of $G_{2}$, there exists a finite subgroup $F_1$ of $G_{1}$ such that $|F_2|$ is a divisor of $|F_1|$.
\end{itemize}
\end{theorem}

It follows that there are continuum many distinct types of countable locally finite groups and each group is finitarily
asymorphic to some direct sum of finite cyclic groups.

The following coarse classification of countable Abelian groups is obtained in \cite{b2}.
\vspace{3 mm}

\begin{theorem} Two countable groups $G$  and $H$ are finitarily
  coarsely equivalent if  and only if the torsion-free ranks of $G$  and $H$
coincide  and $G, H$  are either  both finitely generated or infinitely generated.
\end{theorem}

In particular, any two countable torsion Abelian groups are  finitarily
coarsely equivalent.

For $\kappa$-asymorphisms, we have the following two results.

\begin{theorem}[\cite{PT}] For any uncountable cardinal $\kappa$, any two groups $G,H$ of cardinality $|G|=\kappa=|H|$ are $\kappa$-asymorphic.
\end{theorem}

\begin{theorem}[\cite{b22}]\label{2.3} Let $\kappa$ be a cardinal and $G$ be an  Abelian group of cardinality $|G|\ge\kappa$. The group $G$ is $\kappa$-asymorphic to a free Abelian group. If $\kappa<|G|$ or $\kappa=|G|$ is a singular cardinal, then $G$ is  not not $\kappa$-coarsely equivalent to a free group.  In particular, $G$ is not  $\kappa$-asymorphic to a free group.
\end{theorem}

\begin{theorem}[\cite{b18}]\label{2.5}
For every countable group $G$  there are continuum many
 distinct classes of finitarily coarsely equivalent subsets of $G$.
 \end{theorem}

\subsection{Free coarse groups}
A class $\mathfrak{M}$ of groups is a {\it variety}  if $\mathfrak{M}$ is
closed under  subgroups, products and
 homomorphic images.
We assume that $\mathfrak{M}$ is non-trivial (i.e.
 there exists  $G\in \mathfrak{M}$  such that  $|G|>1$)  and recall that
 the a group  $F_{\mathfrak{M}} (X)$ in the variety $\mathfrak M$ is defined by the following
 conditions: $F_{\mathfrak{M}} (X)\in\mathfrak{M}$, $X\subset F_{\mathfrak{M}} (X)$,
 $X$   generates  $F_{\mathfrak{M}} (X)$
 and every mapping $X \to G$,  $G\in \mathfrak{M}$ can be extended to  homomorphism
 $F_{\mathfrak{M}} (X)\to G$.

For a coarse space $(X, \mathcal{E})$, a {\it free coarse group} $F_{\mathfrak{M}} (X, \mathcal{E})$ is defined as a coarse group
$(F_{\mathfrak{M}} (X), \mathcal{E}^{\prime})$ such that $(X, \mathcal{E})$
is a subballean of
   $(F_{\mathfrak{M}} (X), \mathcal{E}^{\prime})$
   and every coarse mapping
   $(X, \mathcal{E})\to (G, \mathcal{E}^{\prime\prime})$,
   $G\in \mathfrak{M}$, $(G, \mathcal{E^{\prime\prime}})$
   is a coarse group, can be extended to  coarse homomorphism
   $(F_{\mathfrak{M}} (X), \mathcal{E}^{\prime})\to (G, \mathcal{E^{\prime\prime}})$. The definition implies that a free coarse group is unique up to an asymorphism, which is identity on $X$.

The following  theorem is proved in  \cite{b19} with explicit description of the coarse structure of $F_{\mathfrak{M}} (X, \mathcal{E})$.

\begin{theorem}\label{2.6}
For every coarse space  $(X, \mathcal{E})$  and every non-trivial variety $\mathfrak{M}$ of groups,
there exists  a free coarse group  $F_{\mathfrak{M}} (X, \mathcal{E})$.
\end{theorem}

\subsection{Maximality}
 A topological space $X$ with no isolated points is called {\it maximal} if $X$ has an isolated point in any stronger topology. A topological groups $G$ is called {\it maximal} if $G$ is maximal as a topological space.  Every maximal  topological group has an open countable Boolean  subgroup,  and can be constructed using the Martin Axiom.
On the other hand, the existence of a maximal topological group implies the existence of a
$P$-point in  $\omega^{\ast}$, see \cite{10}.

An unbounded coarse space $(X, \mathcal{E})$ is called {\it maximal } if $X$  is bounded in every coarse structure $\mathcal{E}^{\prime}$  such that $\mathcal{E}\subset \mathcal{E}^{\prime}$.
A coarse group $G$  is called {\it maximal } if $G$ is maximal as a coarse space.
If a coarse group $(G, \mathcal{I})$ is maximal then $\{ g ^{2}: g\in G\}$  is bounded in
$(G, \mathcal{I})$,
 and under $CH$ a maximal coarse Boolean group is constructed in \cite{b23}, (see also \cite[Chapter 10]{b24}), but the following question remains open.

\begin{question}\label{2.1} Does there exists a maximal coarse group in ZFC?
\end{question}

\subsection{Normality}
We say  that subsets $Y, Z$ of a coarse space $(X, \mathcal{E})$ are {\it asymptotically disjoint} if, for every $\varepsilon\in \mathcal{E}$ the intersection
 $B(Y,\varepsilon)\cap B(Z,\varepsilon )$ is bounded.

A subset $U\subset X$ of a coarse space $(X,\mathcal E)$ is called a {\em
 asymptotic neighborhood} of a set $A\subset X$ if the sets $A$ and $X\setminus U$ are asymptotically disjoint.

A coarse space $(X, \mathcal{E})  $  is called {\it normal}  if any  asymptotically disjoint subsets $Y,Z\subset X$ have disjoint asymptotic neighborhoods $O_Y$, $O_Z$.

\begin{theorem}[\cite{P03}] For a coarse space $(X,\mathcal E)$ the following conditions are equivalent:
\begin{enumerate}
\item $(X,\mathcal E)$ is normal;
\item for any disjoint and asymptotically disjoint sets $Y,Z\subset X$ there exists  a slowly oscillating function
$f: (X, \mathcal{E})\to   [0,1]$   such that $f(Y)\subset\{0\}$
 and    $f(Z)\subset\{1\}.$
\item for each subballean $Y$ of $X$,  every bounded slowly oscillating function  $f: Y\to\mathbb{R}$  can be extended to a bounded slowly oscillating function on $X$.
\end{enumerate}
\end{theorem}

We recall that a real-valued function $f:X\to\mathbb R$ defined on a coarse space $(X,\mathcal E)$ is {\em slowly oscillating} if for any $\e\in\mathcal E$ and real number $\delta>0$ there exists a bounded set $B\subset X$ such that $\mathrm{diam} f(B(x,\e))<\delta$.

Every metrizable coarse space is normal. More generally, a coarse space is normal if its coarse structure has a linearly ordered base, see \cite{P03}. A partial conversion of this result for products was recently proved in \cite{BP}.

\begin{theorem} If the product $X\times Y$ of two unbounded coarse spaces $X,Y$ is normal, then the coarse space $X\times Y$ has bounded growth and its bornology has a linearly ordered base.
\end{theorem}

A coarse space $(X,\mathcal E)$ is defined to have {\em bounded growth} if there exists a (multi-valued) function $\Phi:X\to\mathcal P_X$ such that for every bounded set $B\subset X$ the union $\bigcup_{x\in B}\Phi(x)$ is bounded and for every entourage $\e\in\mathcal E$ there exists a bounded set $D\subset X$ such that $B(x,\e)\subset\Phi(x)$ for all $x\in X\setminus D$.


\begin{theorem}[\cite{BP}] Let $\kappa$ be an infinite cardinal and $G$ be a group of cardinality $|G|\ge\kappa$, endowed with the coarse structure $\mathcal E_{[G]^{<\kappa}}$, generated by the group ideal $[G]^{<\kappa}$. If the coarse space $(X,\mathcal E_{[G]^{\kappa}})$ is normal (and $G$ is solvable), then $|G|=\kappa$ (and the cardinal $\kappa=|G|$ is regular).
\end{theorem}

\subsection{Coarse structures on topological groups}
A subset $A$ of a topological group $G$ (all topological groups are supposed to be Hausdorff) is called {\it totally bounded} if, for every neighborhood $U$ of the identity, there exists a finite set $F\subset G$ such that
 $A\subseteq F U\cap  UF$.
The group bornology $\mathcal{B}_{\tau}$  of all totally bounded subsets of $(G, \tau)$ defines two (antitall) coarse structures
 $\mathcal{E} _{l}$ and $\mathcal{E} _{r}$ generated by the ball structures  $$\big\{\{(x, y)\in G\times G: y\in\{x\}\cup Bx \}:  B \in \mathcal{B}_{\tau}\big\}\mbox{ and }
 \{\{(x, y)\in G\times G: y\in\{x\}\cup Bx:  B \in \mathcal{B}_{\tau}\big\},$$respectively.

The following questions are from \cite{b6}.

\begin{question}\label{2.3}
Given a group bornology $\mathcal{B}$ on a group $G$,  how can one detect whether there exists a group topology $\tau$  on $G$  such that $\mathcal{B}= \mathcal{B}_{\tau}$?
\end{question}

Let $(G, \tau)$  be a topological group.  We denote by
$\tau^{\sharp}$ the strongest group topology on $G$  such that
$\mathcal{B}_{\tau^{\sharp}}= \mathcal{B}_{\tau}$,
 and say that $(G, \tau)$  is $b$-{\it determined } if
 $\tau^{\sharp}= \tau$.
Clearly,  every discrete group is  $b$-determined.  A totally bounded group $G$ is $b$-determined if and only if
$\tau$  is the maximal totally bounded topology on $G$.

\begin{question}\label{2.4}
Given a topological group $G$,  how can one detect whether  $G$ is
$b$-determined?
\end{question}

For the coarse structures $\mathcal{E }_{l}$,  $\mathcal{E} _{r}$  and slowly oscillating functions on locally compact groups, see \cite{b4}.

\section{ Uniform  groups}

We recall that a family $\mathcal{U}$ of subsets of $X\times  X$  is a {\it uniformity} on a set $X$ if
\begin{itemize}
\item{} $\vartriangle_{X}\subseteq u$ for each $u\in\mathcal{U}$;
\item{} if  $u,v\in \mathcal{U}$ then  $u\cap v \in  \mathcal{U}$;
\item{} if $u\in \mathcal{U}$ and $u\subseteq  v\subset X\times X$  then $v\in \mathcal{U}$.
\item{} for every  $u\in \mathcal{U}$, there exists  $v\in \mathcal{U}$ such that $v\circ   v^{-1} \subseteq u$.
\end{itemize}

A family $\mathcal{F} \subseteq \mathcal{U}$ is called a {\it base} of $\mathcal{U}$ if for every
 $u\in \mathcal{U}$  there exists $v\in \mathfrak{F}$  such that
 $v \subseteq u$.
A set $X$, endowed with a uniformity $\mathcal{U}$, is called a {\it uniform spaces}.

We say that a uniformity $\mathcal{U}$ on a group $G$ is {\it left} ({\em right}) {\em invariant} if $\mathcal{U}$  has a base consisting of   left (right) translation invariant entourages (cf. 2.2).

A filter $\varphi$  on a group $G$  is called a {\it group filter} if for every $A\in \varphi$ there exists $B\in \varphi$  such that $BB^{-1} \in \varphi$. In this case every set $A\in\varphi$ contains the unit $e$ of the group $G$. If $\{\e\}\in\varphi$, then the group filter $\varphi$ is called {\it principal}.

Every group filter $\varphi$  defines two uniformities $\mathcal{L}_{\varphi}$  and
$\mathcal{R}_{\varphi}$    on $G$  with the bases
$$\big\{\{(x, y)\in G\times G: y\in xA \}: A \in \varphi\big\}\mbox{ \ and \ } \big\{\{(x, y)\in G\times G: y\in Ax\}:  A \in \varphi\big\}.$$

\begin{proposition}\label{3.1} A uniformity $\mathcal{U}$ on a group $G$ is right invariant if and only if $\mathcal{U} = \mathcal{R}_{\varphi}$
  for some group filter $\varphi$ on $G$.
\end{proposition}

We recall that a topological group $G$ is {\it balanced}  (= SIN)  if the left and right uniformities on $G$ coincide $(=G$ has a base of invariant neighborhoods of the identity).

\begin{proposition}\label{3.2} A uniformity $\mathcal{U}$ on a group $G$ is left and right invariant if and only if $G$ endowed with the topology generated by the uniformity $\mathcal U$ is a
balanced topological group.
\end{proposition}

\begin{example}\label{3.1}
Let $H$  be a topological group and let $f$  be an arbitrary automorphism  of $H$.
Then the semidirect product $G= H \leftthreetimes \langle f\rangle$  is a right uniform group defined by the filter of neighbourhoods of $e$  in $H$.
If $f$ is discontinuous then $G$  is not left uniform.  Now let $H$ be the Cartesian  product  of infinitely many copies of $\mathbb{Z}_{p}$.
It is easy to  find a discontinuous  automorphism of order 2 of $H$.  Hence, the right uniform group $G$  contains a   compact topological group of index 2 but $G$  is not a topological  group.
\end{example}

We say that a group $G$ is {\it uniformizable}   if there is a non-principal group filter $\varphi$ on $G$ such that $\cap\varphi=\{e\}$.

We recall that a group $G$ is {\it topologizable}  if $G$ admits a non-discrete (Hausdorff) group topology.
Clearly, every  topologizable group is  uniformizable,  but the class of uniformizable groups is much wider  than the class of topologizable groups.
By \cite{b26},  every group can be embedded into some non-topologizable group, and if a group $G$ contains a uniformizable subgroup then $G$ is uniformizable.

Does there exists a non-uniformizable group?
In \cite{b8}, Alexander  Olshansky used the following example to construct a countable non-topologizable group.

\begin{example}\label{3.2} Let $n\ge 2$ be a natural number and $m\ge 665$ be an odd number. Let $A(n,m)$ be the Adian group,  see \cite{b1}.
This group is generated by $n$ elements and has the following properties:
\begin{enumerate}
\item[(a)	] $A(n, m) $ is torsion free;
\item[(b)] the center $C$ of $A(n, m)$ is an  infinite cyclic group,  $C=\langle c\rangle$;
\item[(c)]  $A(n, m)/  C$  is an infinite group of period $m$.
\end{enumerate}
We put $G= A(n, m) / \langle c^{m}\rangle $,   denote by $f: A(n, m)\to G$  the quotient map  and
observe that if $g\in G\setminus \{e\} $  then $g$ or $g^m$ belongs to the set
$\{f(c),  f(c^{2}), \ldots , f(c^{m-1})\}$.
It follows that if $\varphi$  is a group filter on $G$ and $\cap\varphi =\{e\}$ then
$\{e\}\in\varphi$, so $G$  is non-uniformizable.
\end{example}

\begin{question}\label{3.1}
Does every uniformizable group contain a topologizable subgroup?
\end{question}

\begin{question}\label{3.2}
Can one find  a criterion of uniformizability  of countable groups in spirit of Markov's criterion of  topologizability?
\end{question}


\end{document}